\documentclass[12pt,a4paper,english,reqno]{amsart}
\usepackage[a4paper,footskip=1.5em]{geometry}
\usepackage{amsmath,amssymb,amsthm,mathtools,mathrsfs}
\usepackage{adjustbox,tikz,calc,graphics,babel}
\usepackage{csquotes,enumerate,verbatim}
\usepackage[final]{microtype}
\usepackage[numbers]{natbib}
\usetikzlibrary{shapes.misc,calc,intersections,patterns,decorations.pathreplacing}
\usepackage{hyperref}
\hypersetup{colorlinks=true,linkcolor=blue,citecolor=blue}

\allowdisplaybreaks

\theoremstyle{plain}
\newtheorem*{theorem*}{Theorem}
\newtheorem{theorem}{Theorem}[section]
\newtheorem{lemma}[theorem]{Lemma}

\newtheorem{proposition}[theorem]{Proposition}
\newtheorem*{claim*}{Claim}

\newtheorem{problem}[theorem]{Problem}
\theoremstyle{remark}
\newtheorem*{remark}{Remark}

\def\nat{\mathbb{N}}
\def\Z{\mathbb{Z}}

\def\R{\mathbb{R}}
\def\P{\mathbb{P}}
\def\E{\mathbb{E}}
\def\C{\mathcal}

\DeclareMathOperator\V{Var}
\let\emptyset\varnothing
\let\eps\varepsilon

\let\originalleft\left
\let\originalright\right
\renewcommand{\left}{\mathopen{}\mathclose\bgroup\originalleft}
\renewcommand{\right}{\aftergroup\egroup\originalright}

\makeatletter
\def\imod#1{\allowbreak\mkern10mu({\operator@font mod}\,\,#1)}
\makeatother

\def\nr{[n]^{(r)}}
\def\VV{\mathbf{V}}
\def\NN{\mathbf{N}}
\def\MM{\mathbf{M}}

\begin{document}

\title{On the stability of the Erd\H{o}s--Ko--Rado theorem}

\author{B\'{e}la Bollob\'{a}s}
\address{Department of Pure Mathematics and Mathematical Statistics, University of Cambridge, Wilberforce Road, Cambridge CB3\thinspace0WB, UK; {\em and\/}
Department of Mathematical Sciences, University of Memphis, Memphis TN 38152, USA; {\em and\/} London Institute for Mathematical Sciences, 35a South St., Mayfair, London W1K\thinspace2XF, UK.}
\email{b.bollobas@dpmms.cam.ac.uk}

\author{Bhargav Narayanan}
\address{Department of Pure Mathematics and Mathematical Statistics, University of Cambridge, Wilberforce Road, Cambridge CB3\thinspace0WB, UK}
\email{b.p.narayanan@dpmms.cam.ac.uk}

\author{Andrei Raigorodskii}
\address{Lomonosov Moscow State University, Mechanics and Mathematics Faculty, Department of Math. Statistics and Random Processes, Leninskie gory, Moscow, 119991, Russia; {\em and \/} Moscow Institute of Physics and Technology, Faculty of Innovations and High Technology, Institutskiy per., Dolgoprudny, Moscow Region, 141700, Russia.}
\email{mraigor@yandex.ru}

\date{26 March 2014}

\begin{abstract}
Delete the edges of a Kneser graph independently of each other with some probability: for what probabilities is the independence number of this random graph equal to the independence number of the Kneser graph itself? We prove a sharp threshold result for this question in certain regimes. Since an independent set in the Kneser graph is the same as a uniform intersecting family, this gives us a random analogue of the Erd\H{o}s--Ko--Rado theorem.
\end{abstract}

\maketitle

\section{Introduction}
In this note, our aim is to investigate the stability of a central result in extremal set theory due to Erd\H{o}s, Ko and Rado~\citep{EKR} about uniform intersecting families of sets. A family of sets $\C{A}$ is said to be \emph{intersecting} if $A \cap B \neq \emptyset$ for all $A, B \in \C{A}$. We are interested in intersecting families where all the sets have the same size; writing $[n]$ for the set $\{ 1,2,\dots,n\}$ and $\nr$ for the family of all the subsets of $[n]$ of cardinality $r$, the Erd\H{o}s--Ko--Rado theorem asserts that if $\C{A} \subset \nr$ is intersecting and $n \ge 2r$, then $|\C{A}| \le \binom{n-1}{r-1}$ and that equality is only achieved, if $n > 2r$, when $\C{A}$ is a \emph{star}; for $x \in [n]$, the \emph{star centred at $x$} is the family of all the $r$-element subsets of $[n]$ containing $x$. Extending this result, Hilton and Milner~\citep{Hilton} determined, when $n > 2r$, the largest size of a uniform intersecting family not contained entirely in a star. Many extensions of the Erd\H{o}s--Ko--Rado theorem and the Hilton--Milner theorem have since been proved; furthermore, very general stability results about the structure of intersecting families have been proved by Friedgut~\citep{Stab1}, Dinur and Friedgut~\citep{Stab2}, and Keevash and Mubayi~\citep{Stab3}.

Here, we shall investigate a different notion of stability and prove a `sparse random' analogue of the Erd\H{o}s--Ko--Rado theorem which strengthens the Erd\H{o}s--Ko--Rado theorem significantly when $r$ is small compared to $n$. 

To translate the Erd\H{o}s--Ko--Rado theorem to the random setting, it will be helpful to reformulate the theorem as a statement about Kneser graphs. For natural numbers $n,r \in \nat$ with $n\ge r$, the Kneser graph $K(n,r)$ is the graph whose vertex set is  $\nr$ where two $r$-element sets $A,B \in \nr$ are adjacent if and only if $A \cap B = \emptyset$. Observe that a family $\C{A} \subset \nr$ is an intersecting family if and only if $\C{A}$ is an independent set in $K(n,r)$. Writing $\alpha (G)$ for the size of the largest independent set in a graph $G$, the Erd\H{o}s--Ko--Rado theorem asserts that $\alpha (K(n,r)) = \binom{n-1}{r-1}$ when $n \ge 2r$; furthermore, when $n > 2r$, the only independent sets of this size are stars.

Let us now randomly delete the edges of the Kneser graph $K(n,r)$ retaining them with some probability $p$, independently of each other. When is the independence number of this random subgraph equal to $\binom{n-1}{r-1}$? It turns out that when $r$ is much smaller than $n$, an analogue of the Erd\H{o}s--Ko--Rado theorem continues to be true even after we delete practically all the edges of the Kneser graph!

This kind of phenomenon, namely the validity of classical extremal results for surprisingly sparse random structures, has received a lot of attention over the past twenty five years.

Perhaps the first result of this kind in extremal graph theory was proved by Babai, Simonovits, and Spencer~\citep{Triangle1} who showed that an analogue of Mantel's Theorem is true for certain random graphs. Mantel's Theorem states that the largest triangle free subgraph and the largest bipartite subgraph of $K_n$, the complete graph on $n$ vertices, have the same size. Babai, Simonovits, and Spencer proved that the same holds for the Erd\H{o}s-R\'enyi random graph $G(n,p)$ with high probability when $p \ge 1/2 -\delta$ for some absolute constant $\delta >0$. In other words, they show that Mantel's theorem is `stable' in the sense that it holds not only for the complete graph but that it holds \emph{exactly} for random subgraphs of the complete graph as well. Improving upon results of Brightwell, Panagiotou and Steger~\citep{Triangle2}, DeMarco and Kahn~\citep{Triangle3} have recently shown that this phenomenon continues to hold even when the random graph $G(n,p)$ is very sparse; they show in particular that it suffices to take $p \ge C(\log{n} / n)^{1/2}$ for some absolute constant $C>0$, and that this is best-possible up to the value of the absolute constant.

The first such transference results in Ramsey theory were proved by R\"{o}dl and Ruci\'{n}ski~\citep{Ramsey1, Ramsey2} and there have been many related Ramsey theoretic results since; see, for example,~\citep{Ramsey3, Ramsey4, Ramsey5}. 

Phenomena of this kind have also been observed in additive combinatorics. Roth's theorem~\citep{Roth1}, a central result in additive combinatorics, states that for every $\delta >0$ and all sufficiently large $n$, every subset of $[n]=\lbrace 1, 2, \dots ,n \rbrace$ of density $\delta$ contains a three-term arithmetic progression. Kohayakawa, R\"{o}dl and {\L}uczak~\citep{Roth2} proved a random analogue, showing that such a statement holds not only for $[n]$ but also, with high probability, for random subsets of $[n]$ of density at least $Cn^{-1/2}$, where $C>0$ is an absolute constant. 

Another classical result in additive combinatorics, due to Diananda and Yap~\citep{sumfree0}, is that the largest sum-free subset of $\Z _{2n}$ is the set of odd numbers. Balogh, Morris and Samotij~\citep{sumfree1} proved that the same is true of random subsets of $\Z _{2n}$ of density at least $(1+\eps)(\log{n}/{3n})^{1/2}$ with high probability (for any fixed $\eps > 0$ and $n$ sufficiently large), and also that this no longer the case when the density is less than $(1-\eps)(\log{n}/{3n})^{1/2}$. Thus, there is a \emph{sharp threshold} at $(\log{n}/{3n})^{1/2}$ for the stability of this extremal result; an extension of this sharp threshold result to all even-order Abelian groups has recently been proved by Bushaw, Collares Neto, Morris and Smith~\citep{sumfree2}.

Perhaps the most striking application of such transference principles in additive combinatorics is the Green--Tao theorem~\citep{primes} on primes in arithmetic progressions.

These results constitute a tiny sample of the large number of beautiful results which have been proved in this setting. Very general transference theorems have been proved by Conlon and Gowers~\citep{Conlon} and Schacht~\citep{Schacht}, and more recently, by Balogh, Morris and Samotij~\citep{cont1} and Saxton and Thomason~\citep{cont2}. We refer the interested reader to the surveys of {\L}uczak~\citep{survey2} and R\"{o}dl and Schacht~\citep{survey1} for a more detailed account of such results.

Returning to the question at hand, our aim in this paper, as we remarked before, is to investigate the independence number of random subgraphs of $K(n,r)$; for related work on the independence number of random \emph{induced} subgraphs of $K(n,r)$, see the paper of Balogh, Bohman and Mubayi~\citep{EKR_vertex}. Let $K_p(n,r)$ denote the random subgraph of $K(n,r)$ obtained by retaining each edge of $K(n,r)$ independently with probability $p$. The main question of interest is the following.

\begin{problem}
\label{ekr14-1-mainques}
For what $p > 0$ is $\alpha (K_p(n,r)) = \binom{n-1}{r-1}$ with high probability? 
\end{problem}
For constant $r$ and $n$ sufficiently large, a partial answer  was provided by Bogolyubskiy, Gusev, Pyaderkin and Raigorodskii~\citep{distance1, distance2}: they studied random subgraphs of $K(n,r,s)$, where $K(n,r,s)$ is the graph whose vertex set is $\nr$ where two $r$-element sets $A,B \in \nr$ are adjacent if and only if $|A \cap B| = s$; in the case $s=0$ (which corresponds to the Kneser graph), they established that $\alpha (K_{1/2}(n,r)) = (1+o(1))\binom{n-1}{r-1}$ with high probability.

We shall do much more and answer Question~\ref{ekr14-1-mainques} exactly when $r$ is small compared to $n$ (more precisely, when $r = o(n^{1/3}))$. To state our result, it will be convenient to define the threshold function
\begin{equation}\label{ekr14-1-thresh}
p_c(n,r) = \frac{(r+1) \log{n} - r\log{r}}{\binom{n-1}{r-1}}.
\end{equation}
As we shall see, this is the threshold density at which one expects to find a vertex in $K_p(n,r)$ which has no edges to a maximal independent set of the original Kneser graph $K(n,r)$. With this definition in place, we can now state our main result.

\begin{theorem}
\label{ekr14-1-mainthm}
Fix a real number $\eps>0$ and let $r=r(n)$ be a natural number such that $2 \le r(n) = o(n^{1/3})$. Then as $n \to \infty$, 
\[ 
\P \left(\alpha (K_p(n,r)) = \binom{n-1}{r-1} \right) \to
\begin{cases}
   1 &\mbox{if } p \ge (1+\eps)p_c(n,r)\\
   0 &\mbox{if } p \le (1-\eps)p_c(n,r).\\
   \end{cases}
\]
Furthermore, when $p \ge (1+\eps)p_c$, with high probability, the only independent sets of size $\binom{n-1}{r-1}$ in $K_p(n,r)$ are the trivial ones, namely, stars.
\end{theorem}

The rest of this paper is organised as follows. We establish some notation and collect together some standard facts in Section~\ref{ekr14-1-estimates}. Most of the work involved in proving Theorem~\ref{ekr14-1-mainthm} is in establishing the upper bound on the critical density; we do this in Section~\ref{ekr14-1-upper}. We complete the proof of Theorem~\ref{ekr14-1-mainthm} by proving a matching lower bound in Section~\ref{ekr14-1-lower}. We conclude with some discussion in Section~\ref{ekr14-1-conc}.

\section{Preliminaries}\label{ekr14-1-estimates}

\subsection{Notation}
Given $x \in [n]$ and $\C{A} \subset \nr$, we write $\C{S}_x$ for the star centred at $x$, and $\C{A}_x$ for the subfamily of $\C{A}$ consisting of those sets (of $\C{A}$) that contain $x$, i.e., $\C{A}_x=\C{A} \cap \C{S}_x$. The maximum degree $d(\C{A})$ of a family $\C{A} \subset \nr$ is defined to be the maximum cardinality, over all $x \in [n]$, of the subfamily $\C{A}_x$, and we write $e(\C{A})$ for the number of edges induced by $\C{A}$ in $K(n,r)$. Since any pair of intersecting sets $A, B \in \C{A}$ both belong to at least one subfamily $\C{A}_x$, we get the following estimate for $e(\C{A})$ which is useful when the maximum degree of $\C{A}$ is small.

\begin{proposition}\label{ekr14-1-edges}
For any $\C{A} \subset \nr$,
\[ e(\C{A}) \ge \binom{|\C{A}|}{2} - \sum_{x \in [n]} \binom{|\C{A}_x|}{2}. \eqno\qed \]
\end{proposition}

To ease the notational burden, in the rest of this paper, we shall write $\VV = \binom{n}{r}$ for the size of $\nr$, and $\NN = \binom{n-1}{r-1}$ for the size of a star. Also, given $x\in [n]$ and a set $A \in \nr$ not containing $x$, we shall write $\MM = \binom{n-r-1}{r-1}$ for the number of sets of $\C{S}_x$ disjoint from $A$.

A word on asymptotic notation; we use the standard $o(1)$ notation to denote any function that tends to zero as $n$ tends to infinity. Here and elsewhere, the variable tending to infinity will always be $n$ unless we explicitly specify otherwise.

\subsection{Estimates} Next, we collect some standard estimates that we shall use repeatedly; for ease of reference, we list them as propositions below.

Let us start with a weak form of Stirling's approximation for the factorial function.

\begin{proposition}\label{ekr14-1-sterling}
For all $n \in \nat$,
\[ \sqrt{2\pi n} \left( \frac{n}{e} \right)^n \le n! \le e^{1/12n} \sqrt{2\pi n} \left( \frac{n}{e} \right)^n. \eqno\qed\]
\end{proposition}

In fact, the following crude bounds for the binomial coefficients will often be sufficient for our purposes.

\begin{proposition}
For all $n, r \in \nat$,
\[ \left(\frac{n}{r} \right) ^{r} \le \binom{n}{r} \le \frac{n^r}{r!} \le \left(\frac{en}{r} \right) ^{r}. \eqno\qed \]
\end{proposition}

Also, we will need the following standard inequality concerning the exponential function.

\begin{proposition}
For every $x\in \R$ such that $|x| \le 1/2$,
\[  e^{x - x^2} \le 1 + x \le e^x. \eqno\qed \]
\end{proposition}

Although our last proposition is also very simple, we prove it here for the sake of completeness. Recall that $\NN = \binom{n-1}{r-1}$ and $\MM = \binom{n-r-1}{r-1}$.

\begin{proposition}\label{ekr14-1-diff}
If $r = r(n) = o(n^{1/2})$, then $\NN - \MM = o(\NN)$. Furthermore, if $r = o(n^{1/3})$, then $\NN - \MM = o(\NN/r)$.
\end{proposition}

\begin{proof}
Both claims follow from the observation that 
\begin{align*}
\NN - \MM &= \binom{n-1}{r-1} - \binom{n-r-1}{r-1}\\
      &= \sum_{i = 1} ^r \binom{n-i}{r-1} - \binom{n-i-1}{r-1}\\ 
      &= \sum_{i = 1} ^r \binom{n-i-1}{r-2}\\
      &\le r\binom{n-2}{r-2} = \frac{r(r-1)}{n-1} \NN.\qedhere
\end{align*}
\end{proof}

\section{Upper bound for the critical threshold}\label{ekr14-1-upper}

We now turn to our proof of Theorem~\ref{ekr14-1-mainthm}. In this section, we shall bound the critical threshold from above, i.e., we shall prove that a random analogue of the Erd\H{o}s--Ko--Rado theorem holds if $p > (1+\eps)p_c(n,r)$ where $p_c(n,r)$ is given by~\eqref{ekr14-1-thresh}.

Let us remind the reader before we begin that for us, a star in the Kneser graph is a maximal trivial intersecting family of sets (and this should not be confused with the graph-theoretic notion of a star).

\begin{proof}[Proof of the upper bound in Theorem~\ref{ekr14-1-mainthm}]
Let $0 < \eps < 1/2$ and set $p = p(n) = (1+\eps)p_c(n,r)$. We shall prove that with high probability, the independence number of $K_p(n,r)$ is $\NN$, and that furthermore, the only independent sets of size $\NN$ in $K_p(n,r)$ are stars. Since we are working with monotone properties, it suffices to prove this result for $\eps$ small enough, so we lose nothing by assuming $0 < \eps < 1/2$.

For each $i \ge 1$, let $X_i$ be the number of families $\C{A} \subset \nr$ inducing an independent set in $K_p(n,r)$ such that $|\C{A}| = \NN$ and $d(\C{A}) = \NN-i$. Also, let $Y$ be the number of independent families $\C{A} \subset \nr$ such that $|\C{A}| = \NN+1$ and $d(\C{A}) = \NN$; in other words, independent families of size $\NN+1$ which contain an entire star.

Our aim is to show that with high probability, the random variables defined above are all equal to zero. This then implies the lower bound on the critical threshold; since every $X_i$ is equal to zero, every independent set in $K_p(n,r)$ of cardinality at least $\NN$ must contain an entire star, and since $Y$ is also equal to zero, the only independent sets of cardinality at least $\NN$ are stars.

We start by computing $\E[Y]$. We know that for any star $\C{S}$, any $A \in \nr \setminus \C{S}$ is disjoint from $\MM$ elements of $\C{S}$, so
\begin{equation}\label{ekr14-1-EY}
\E[Y] =  \binom{n}{1}\binom{\VV-\NN}{1}(1-p)^{\MM}. 
\end{equation}

When $r = o(n^{1/3})$ (indeed, when $r = o(n^{1/2})$), we know from Proposition~\ref{ekr14-1-diff} that $\MM = (1+o(1))\NN$. Since $p = (1+\eps)((r+1)\log{n} - r\log{r})/\NN$, we see that
\begin{align*}
\E[Y] &\le  n\VV(1-p)^{(1+o(1))\NN}\\
	  &\le n\left( \frac{en}{r} \right)^r \exp{\left((-1+o(1))p\NN\right)} \\
	  &\le n\left( \frac{en}{r} \right)^r \exp{\left( (1 + \eps + o(1))(r\log{r} - (r+1)\log{n}) \right)}\\
	  &\le \left( \frac{er}{n} \right)^{(\eps + o(1))r}  \le n^{-(\eps + o(1))2r/3} = o(1).
\end{align*}

By Markov's inequality, we know that $\P(Y>0) \le \E[Y]$ and it follows that $Y$ is zero with high probability.

We now turn our attention to the $X_i$. To keep our argument simple, we distinguish three cases: we first deal with small values of $i$ where the $X_i$ count families of very large maximum degree, then we consider families of large (but not huge) maximum degree, and in the final case, we deal with families of small maximum degree.

\textbf{Case 1: Very large maximum degree.}
Unfortunately, when $i$ is small, it is not true that $\E[X_i]$ goes to zero as $n$ grows. For constant $i$, $\E[X_i] \ge  n\binom{\NN}{i}\binom{\VV - \NN}{i}(1-p)^{(i+o(1))\NN}$. When $r=3$ and $i=2$ for example, it follows that
\begin{align*}
\E[X_2] &\ge  n\binom{\binom{n-1}{2}}{2}\binom{\binom{n}{3}-\binom{n-1}{2}}{2}(1-p)^{(2+o(1))\NN}\\
	  &\ge n ^ {o(1)} \frac{n^{11}}{n^{8(1+\eps)}}\ge n ^ {3 - 8\eps +o(1)},
\end{align*}
which grows with $n$ when $\eps$ is small enough. However, if we compute $\V {[X_2]}$, we are encouraged to find that $\V {[X_2]} / \E [X_2]^2$ is bounded away from zero; indeed, we observe similar behaviour for any fixed value of $i$ and larger $r$ as well. We therefore adopt a different strategy to bound $\P( X_i > 0)$ for small $i$. 

For $j \ge i$, let $X_{i,j}$ be the number of families $\C{A} \subset \nr$ inducing a \emph{maximal} independent set in $K_p(n,r)$ such that $d(\C{A}) = \NN-i$ and $|\C{A}| = \NN + j - i$. If $X_i > 0$, then clearly $X_{i,j}>0$ for some $j \ge i$. To compute $\E[X_{i,j}]$, we note that any family $\C{A}$ counted by $X_{i,j}$ can be described by specifying a star $\C{S}$, a subfamily $\C{A}_1 \subset \C{S}$ of $i$ sets missing from $\C{S}$, and another family $\C{A}_2$ of cardinality $j$ disjoint from $\C{S}$ such that 
\begin{enumerate}
\item all the edges between $\C{S}\setminus \C{A}_1$ and $\C{A}_2$ in $K(n,r)$ are absent in $K_p(n,r)$ (since $\C{A}$ is independent), and
\item each set in $\C{A}_1$ is adjacent to at least one set in $\C{A}_2$ in $K_p(n,r)$ (because $\C{A}$ is a maximal independent set).
\end{enumerate} 

The number of edges between $\C{S}\setminus \C{A}_1$ and $\C{A}_2$ is at least $j(\MM - i)$ since any set not in a star is adjacent to precisely $\MM$ sets in the star in $K(n,r)$. Also, the probability that a set in $\C{A}_1$ has a neighbour in $\C{A}_2$ in $K_p(n,r)$ is at most $jp$. Therefore, we have
\[
\E[X_{i,j}] \le  n\binom{\NN}{i} \binom{\VV}{j}(1-p)^{j(\MM - i)} (jp)^i.
\]

We look at the ratio of the upper bounds for $\E[X_{i,j+1}]$ and $\E[X_{i,j}]$ above and note that this ratio is at most
\[ \VV(1-p)^{\MM-i}(1+1/j)^i.\] When $1 \le i \le \eps \NN/2$ and $j \ge i$, we see, using the fact that $\MM = (1 + o(1))\NN$, that
\begin{align*}
\VV(1-p)^{\MM-i} (1+1/j)^i &\le e\left( \frac{en}{r} \right)^r\exp{\left( -(1 + \eps/2 -\eps^2/2 + o(1))p_c(n,r)\NN \right)} \\
&\le e^{r+1} \left(\frac{r}{n}\right)^{\eps r/ 5} = o(1).
\end{align*}
Consequently, we have 
\begin{align*}
\P[X_i > 0] &\le \sum_{j \ge i} \E[X_{i,j}] \le  2 n\binom{\NN}{i} \binom{\VV}{i}(1-p)^{i(\MM - i)} (ip)^i\\
&\le 2n \left( \frac{e\NN}{i} \right)^i \left( \frac{e\VV}{i} \right)^i \exp{\left( -i(1 + \eps/5)p_c(n,r)\NN\right)} \left(\frac{i(r+1)\log n}{\NN}\right)^i\\
&\le 2e^{2i}n ((r+1)\log n)^i \left( \frac{\VV}{i} \right)^i \exp{\left( -i(1 + \eps/5)p_c(n,r)\NN\right)}\\
&\le 2 \left(\frac{e^{r+2}(r+1)\log n}{i}\left(\frac{r}{n}\right)^{\eps r/5}\right)^i \le 2 \left(e^{r+2}(r+1)\log n\left(\frac{r}{n}\right)^{\eps r/5}\right)^i.
\end{align*}

Summing this estimate for $i\le \eps \NN/2$, we get 
\begin{align*}
\sum_{i=1}^{\eps \NN /2} \P(X_i > 0) &\le \sum_{i=1}^{\eps \NN /2} 2 \left(e^{r+2}(r+1)\log n\left(\frac{r}{n}\right)^{\eps r/5}\right)^i\\
&\le 4 \left(e^{r+2}(r+1)\log n\left(\frac{r}{n}\right)^{\eps r/5}\right) = o(1),
\end{align*}
so by the union bound, with high probability, for each $1 \le i \le \eps \NN/2$, the random variable $X_i$ is zero.

\textbf{Case 2: Large maximum degree.} Next, we consider the $X_i$ with 
\[\eps \NN/2 < i \le \NN\left(1 - \frac{1-\eps/2}{r+1}\right).\]

As noted earlier, for any star $\C{S}$, the number of edges in $K(n,r)$ between a set $A \in \nr \setminus \C{S}$ and a family $\C{A} \subset \C{S}$ is at least $|\C{A}| - (\NN - \MM)$. We know from Proposition~\ref{ekr14-1-diff} that $\NN - \MM = o(\NN/r)$ when $r = o(n^{1/3})$; consequently, it follows that if $\C{A}\subset \nr$ has cardinality $\NN$ and $d(\C{A}) \ge (1-\eps/2)\NN/(r+1)$, then  $e(\C{A}) \ge (1+o(1))d(\C{A})(\NN - d(\C{A}))$.

To simplify calculations, let us define $\alpha$ by setting $i = \alpha \NN = \alpha r\VV/n$, where \[\eps /2 < \alpha \le (r + \eps/2)/(r+1).\] In this range, we see that
\begin{align*}
\E[X_i] &\le  n\binom{\NN}{i} \binom{\VV}{i}(1-p)^{(1+o(1))i(\NN - i)}\\
	  &\le n \left( \frac{e}{\alpha} \right)^{\alpha \NN} \left(\frac{e n}{r \alpha}\right)^{\alpha \NN} \exp{\left( -(1 + \eps + o(1))\alpha(1-\alpha)p_c(n,r)\NN^2 \right)} \\
	  &\le n \left( \frac{n r^{(1+\eps + o(1))(1-\alpha)r}}{r n^{(1+\eps + o(1))(1-\alpha)(r+1)}} \right)^{\alpha \NN}\le n \left( \frac{r}{n} \right)^{(\eps^2/4 - \eps^3/4 + o(1)) \NN}.
\end{align*}
The last two inequalities above are obtained by first collecting the $O(1)$ terms in the bound into the $o(1)$ terms in the exponent, and by then using the bounds on $\alpha$. It follows that 
\[ \sum_{i=\eps \NN /2}^{\frac{(2r+\eps)\NN}{2(r+1)}} \P(X_i > 0) \le n \NN \left( \frac{r}{n} \right)^{(\eps^2/4 - \eps^3/4 + o(1)) \NN} = o(1), \]
so with high probability, for each $ \eps \NN/2 < i \le (r+\eps/2)\NN/(r+1)$, the random variable $X_i$ is zero.

\textbf{Case 3: Small maximum degree.} We shall complete the proof of the lower bound by showing that 
\[  \sum_{i > \frac{(2r+\eps)\NN}{2(r+1)}} \E[X_i] = o(1).\]

It turns out that in this range of $i$, somewhat surprisingly, it is significantly easier to deal with the case where $r$ tends to infinity with $n$ as opposed to the case where $r$ is small. 

Suppose first that $r \ge \log{n}$. Then \[\frac{(1-\eps/2)}{r+1} < \frac{1}{r+4}\] for all large enough $n$. Observe that subgraph of $K(n,r)$ induced by a family $\C{A}$ of cardinality $\NN$ has minimum degree at least $\NN - rd(\C{A})$ and consequently, if  $d(\C{A}) < \NN/(r+4)$, then \[e(\C{A}) \ge \frac{\NN}{2}\left(\NN - \frac{r\NN}{r+4}\right) = \frac{2\NN^2}{r+4}.\]	
In this case, it follows that
\begin{align*}
 \sum_{i > \frac{(2r+\eps)\NN}{2(r+1)}} \E[X_i] &\le \binom{\VV}{\NN}(1-p)^{2\NN^2/(r+4)}\\
 &\le \left(\frac{en}{r} \right)^\NN \left(\frac{r}{n}\right)^{2r\NN/(r+4)}\\
 &\le \left(\frac{er}{n}\right)^{(1+o(1))\NN} = o(1) 
\end{align*}
which completes the proof when $r \ge \log{n}$.

Next, suppose that $r \le \log{n}$. When $r \le \log{n}$, it is not necessarily true (if $r = O(1)$ and $\eps$ is sufficiently small, for instance) that $(1-\eps/2)/(r+1) < 1/(r+4)$. It turns out that in this case, we need a more careful estimate.

For a family $\C{A} \subset \nr$ and each $x \in [n]$, define $\alpha_x = |\C{A}_x|/\NN$. Note that $\sum_{x=1}^n \alpha_x = r$. Recall that Proposition~\ref{ekr14-1-edges} tells us that 
\[ e(\C{A}) \ge \binom{|\C{A}|}{2} - \sum_{x\in [n]} \binom{|\C{A}_x|}{2} \ge \binom{\NN}{2}\bigg(1-\sum_{x \in [n]} \alpha_x^2\bigg).\]

Let $\C{A} \subset \nr$ be such that $|\C{A}| = \NN$ and $d(\C{A}) < (1-\eps/2)\NN/(r+1)$. For such a family $\C{A}$, let $D=D_{\C{A}}$ be the set of $x\in [n]$ such that $\alpha_x \ge (\log{n})^{-2}$. Since $\sum_{x=1}^n \alpha_x = r$, we see that $|D| \le r(\log{n})^{2}\le (\log{n})^3$.

\begin{lemma}\label{ekr14-1-condexp}
Fix $D=D_{\C{A}}$ and the values of $|\C{A}_x|$ for $x \in D$. Subject to these restrictions, the expected number of families $\C{A} \subset \nr$ of maximum degree at most $(1-\eps/2)\NN/(r+1)$ which induce independent sets in $K_p(n,r)$ is at most $(r/n)^{(3/10+o(1))\NN}$.
\end{lemma}
\begin{proof}
Since $\sum_{x=1}^n \alpha_x = r$, it follows (by convexity, for example) that $\sum_{x \in [n] \setminus D} \alpha_x^2$ is at most $r(\log{n})^{-2} \le (\log{n})^{-1} = o(1)$. Consequently, 
\[ e(\C{A}) \ge \frac{\NN^2}{2} \bigg(1 + o(1) -\sum_{x \in [n]} \alpha_x^2\bigg) \ge \frac{\NN^2}{2} \bigg(1 + o(1) -\sum_{x \in D} \alpha_x^2\bigg),\]
so the probability that a family $\C{A}$ as in the statement of the lemma induces an independent set is at most
\begin{align}\label{ekr14-1-prob}
(1-p)^{e(\C{A})} &\le \exp{ \bigg( -\frac{p\NN^2}{2} \bigg(1 + o(1) -\sum_{x \in D} \alpha_x^2\bigg)  \bigg)}\notag\\
&\le \left( \frac{r^{(1+o(1))r}}{n^{(1+o(1))(r+1)}} \prod_{x\in D} \left(\frac{n^{r+1}}{r^r} \right)^{\alpha_x^2} \right)^{\NN/2}.
\end{align}

Next, we bound the number of ways in which we can choose $\C{A}$ as in Lemma~\ref{ekr14-1-condexp}. Using the fact that $r \le \log{n}$ and $|D| \le (\log{n})^{3}$, we first note that
\begin{align*}
 \NN \ge \Bigl \vert \bigcup_{x \in D} \C{A}_x \Bigr \vert &\ge \sum_{x \in D} |\C{A}_x| - \sum_{\substack{x,y \in D\\ x < y}} |\C{A}_x \cap \C{A}_y|\\
 &\ge \sum_{x \in D} |\C{A}_x| - |D|^2\binom{n-2}{r-2}\\
 &\ge \left( \sum_{x \in D} \alpha_x - \frac{|D|^2r}{n} \right)\NN \ge \bigg( \sum_{x \in D} \alpha_x + o(1) \bigg)\NN.
\end{align*}

It follows that 
\begin{equation} \label{ekr14-1-alphasum}
\sum_{x \in D} \alpha_x \le 1+o(1) < 1 + 1/10
\end{equation}
and 
\[|\C{A} \setminus (\cup_{x \in D} \C{A}_x) | < \NN\bigg(1+1/5 - \sum_{x \in D}\alpha_x\bigg).\] 
(Here, the choice of the constants $1/10$ and $1/5$ was arbitrary; any two sufficiently small constants would have sufficed.) Hence, the number of ways to choose $\C{A}$ is at most
\begin{align}\label{ekr14-1-number}
\binom{\VV}{\NN(6/5 - \sum_{x \in D}\alpha_x)}\prod_{x\in D} \binom{\NN}{\alpha_x \NN} & \le \left(\frac{10en}{r}\right)^{\NN(6/5 - \sum_{x \in D}\alpha_x)} \prod_{x\in D} \left(\frac{e}{\alpha_x}\right)^{\alpha_x\NN} \notag\\
&\le 100^\NN \left( \frac{n}{r} \right)^{6\NN/5} \prod_{x \in D} \left( \frac{r}{\alpha_x n } \right)^{\alpha_x \NN}\notag\\
&\le \left( \frac{n}{r} \right)^{(6/5+o(1))\NN} \prod_{x \in D} \left( \frac{r}{\alpha_x n } \right)^{\alpha_x \NN}.
\end{align}

From~\eqref{ekr14-1-prob} and~\eqref{ekr14-1-number}, we conclude that the expected number of independent families $\C{A}$ as in the lemma is at most
\[
\left( \frac{r^{(1 + o(1))r/2 - 6/5}}{n^{(1 + o(1))(r+1)/2 - 6/5}} \right)^\NN \prod_{x \in D} \left( \left( \frac{r}{\alpha_x n } \right)\left(\frac{n^{r+1}}{r^r} \right)^{\alpha_x / 2} \right)^{\alpha_x\NN}
\]

Now, note that$ (r/\alpha n) (n^{r+1}/r^r)^{\alpha / 2} < 1$ whenever $(\log{n})^{-2} \le \alpha < (1-\eps/2)/(r+1)$. Indeed, observe that the function $f$ defined on the positive reals by 
\[ f(\alpha) =  \frac{\alpha((r+1)\log{n} - r\log{r})}{2} - \log{\alpha} +\log{(r/n)} \] 
is convex; so to check that $f(\alpha) < 0 $ when  $(\log{n})^{-2}\le \alpha \le (1-\eps/2)/(r+1)$, it suffices to check that $f((\log{n})^{-2}) < 0$ and $f((1-\eps/2)/(r+1)) < 0$ and both conditions hold for all sufficiently large $n$ when $r \le \log{n}$. 

Therefore, we conclude that the expected number of independent families $\C{A}$ as in the lemma is at most
\[ \left( \frac{r^{(1 + o(1))r/2 - 6/5}}{n^{(1 + o(1))(r+1)/2 - 6/5}} \right)^\NN \le \left( \frac{r^{(1 + o(1))(r+1)/2 - 6/5}}{n^{(1 + o(1))(r+1)/2 - 6/5}} \right)^\NN \le \left( \frac{r}{n} \right)^{(3/10+o(1))\NN},
\]
where the last inequality above follows from the fact that $(r+1)/2 - 6/5 \ge 3/10$ for all $r\ge 2$. This completes the proof of Lemma~\ref{ekr14-1-condexp}.
\end{proof}

Recall that if $r \le \log{n}$ and $d(\C{A}) < (1-\eps/2)\NN/(r+1)$, then $|D_{\C{A}}| \le (\log{n})^{3}$. So the number of choices for the set $D_{\C{A}}$ is clearly at most 
\begin{equation}\label{ekr14-1-choice1}
\sum_{j = 0}^{(\log{n})^{3}} \binom{n}{j} \le (\log{n})^{3} \binom{n}{(\log{n})^{3}}. 
\end{equation} 
We know from~\eqref{ekr14-1-alphasum} that the values  $|\C{A}_x|$ for $x\in D_{\C{A}}$ satisfy 
\[ \sum_{x\in D} |\C{A}_x| \le 11\NN/10,\] 
so the number of ways of selecting the values of $|\C{A}_x|$ is at most 
\begin{equation}\label{ekr14-1-choice2}
\binom{11\NN/10 + (\log{n})^{3}+1}{(\log{n})^{3}} \le (2\NN)^{(\log{n})^{3}}.
\end{equation}
From Lemma~\ref{ekr14-1-condexp}, we conclude using~\eqref{ekr14-1-choice1} and~\eqref{ekr14-1-choice2} that
\begin{equation}\label{ekr14-1-EXsmall} 
\sum_{i > \frac{(2r+\eps)\NN}{2(r+1)}} \E[X_i] \le (\log{n})^{3}n^{(\log{n})^{3}} (2\NN)^{(\log{n})^{3}} \left( \frac{r}{n} \right)^{(3/10+o(1))\NN}.
\end{equation}
It is easy to check that the right-hand side of~\eqref{ekr14-1-EXsmall} is $o(1)$ for every $2 \le r \le \log{n}$. Hence, with high probability, for each $ i > (r+\eps/2)\NN/(r+1)$, the random variable $X_i$ is zero; this completes the proof of the lower bound.
\end{proof}

\begin{remark}
A more careful analysis can be used to show that for large $r$, i.e., when $r$ tends to infinity with $n$, it is sufficient to take $\eps$ to be greater than $6/r$, say, as opposed to a small fixed constant. 
\end{remark}

\section{Lower bound for the critical threshold}\label{ekr14-1-lower}
As in the previous section, let $Y$ be the number of independent families in $K_p(n,r)$ of size $\NN + 1$ which contain an entire star.

\begin{proof}[Proof of the lower bound in Theorem~\ref{ekr14-1-mainthm}]
Turning to the lower bound, we shall assume that $p = (1-\eps)p_c(n,r)$ for some fixed real number $\eps > 0$ and we show using a simple second moment calculation that $Y > 0$ with high probability; consequently, the independence number of $K_p(n,r)$ is at least $\NN+1$.

Recall~\eqref{ekr14-1-EY} which says that 
\[ \E[Y] =  \binom{n}{1}\binom{\VV-\NN}{1}(1-p)^{\MM}.\]
Note that $\NN=o(\VV)$ when $r = o(n^{1/3})$; it follows that
\begin{align*}
\E[Y] &\ge  (1+o(1))n\VV(1-p)^\NN\\
	  &\ge  (1+o(1)) \frac{n^{r+1}}{r!} \exp{\left(-(p + p^2)\NN\right)} \\
	  &\ge \frac{n^{r+1}}{r!} \exp{\left( (1 - \eps + o(1))(r\log{r} - (r+1)\log{n}) \right)}\\
	  &\ge \left( \frac{n}{r} \right)^{(\eps + o(1))r},
\end{align*}
so $\E[Y] \to \infty$ when $p = (1-\eps)p_c(n,r)$. 

Therefore, to show that $Y>0$ with high probability, it suffices to show that $\V [Y] = o(\E[Y]^2)$ or equivalently, that $\E[(Y)_2] = (1+o(1))\E[Y]^2$, where $\E[(Y)_2] = \E[ Y(Y-1)]$ is the second factorial moment of $Y$.

Note that 
\[
\E[(Y)_2] = \sum_{x,y,A,B} \P \left( \C{S}_x \cup \{A\} \, \mbox{and}\, \C{S}_y \cup \{B\} \,\mbox{are independent} \right),\]
the sum being over ordered $4$-tuples $(x,y,A,B)$ with $x,y \in [n]$, $A \in \nr \setminus \C{S}_x$ and $B \in \nr \setminus \C{S}_y$ such that $(x,A) \neq (y, B)$. Now, observe that 
\begin{align*}
\sum_{x \neq y} \P \left( \C{S}_x \cup \{A\} \, \mbox{and}\, \C{S}_y \cup \{B\} \,\mbox{are independent} \right) &\le (n^2) (\VV-\NN)^2 (1-p)^{(2-o(1))\MM}\\
&= (1+o(1))\E[Y]^2,
\end{align*}
and
\begin{align*}
\sum_{x = y, A \neq B} \P \left( \C{S}_x \cup \{A\} \, \mbox{and}\, \C{S}_y \cup \{B\} \,\mbox{are independent} \right) &\le n(\VV-\NN)^2 (1-p)^{2\MM}\\
&= o(\E[Y]^2).
\end{align*}

By Chebyshev's inequality, we conclude that $Y>0$ with high probability, so the independence number of $K_p(n,r)$ is at least $\NN+1$.
\end{proof}

\section{Conclusion}\label{ekr14-1-conc}
The condition $r=o(n^{1/3})$ in our results seems somewhat artificial; we would expect the same formula for the critical threshold to hold for much larger $r$ as well. We suspect that this formula in fact gives the exact value of the critical threshold when $r=o(n)$ but we are unable to prove this presently. 

The size of the critical window also merits study. As we remarked earlier, our proof (for large $r$) works even when we are a factor of $(1 + 6/r)$ away from the critical threshold; it is possible that the critical window is much smaller and it is an interesting problem to determine the size of the critical window precisely.

Of course, one would be interested to know what happens for larger $r$ as well. When $r/n$ is bounded away from $1/2$, we suspect it should be possible to demonstrate stability of the Erd\H{o}s--Ko--Rado theorem at $p=1/2$, say. Perhaps the most interesting question though is the case $n = 2r+1$; it would be interesting to determine if a stability result is true for any probability $p$ bounded away from $1$. We hope to return to these questions in future work.

\section*{Acknowledgements}
The first author would like to acknowledge support from EU MULTIPLEX grant 317532 and NSF grant DMS-1301614. The second author is supported by grant 12-01-00683 of the Russian Foundation for Basic Research, grant MD-6277.2013.1 of the Russian President, and grant NSh-2519.2012.1 supporting Leading Scientific Schools of Russia. 

Some of the research in this paper was carried out while the third author was a visitor at the University of Cambridge and later, while the second author was a visitor at the University of Memphis. The third author is grateful for the hospitality of the University of Cambridge and the second author is grateful for the hospitality of the University of Memphis.

\bibliographystyle{amsplain}
\bibliography{ekr_stability}

\end{document}